\documentclass[final,1p,times]{elsarticle}

\usepackage{amssymb}
\usepackage{amsthm}
\usepackage{amscd}
\usepackage{bm}
\usepackage{amsmath}
\usepackage{amsfonts}
\usepackage{graphicx}
\newtheorem{theorem}{Theorem}[section]
\newtheorem{remark}[theorem]{Remark}
\newtheorem{example}[theorem]{Example}

\newtheorem{lemma}{Lemma}[section]

\newtheorem{definition}{Definition}[section]
\usepackage{mathrsfs}
\usepackage{titletoc}
\usepackage{color}

\newcommand{\ra}{\rightarrow}

\newcommand{\f}{\frac}
\newcommand{\g}{\gamma}

\newcommand{\be}{\begin{equation}}
\renewcommand{\ra}{\rightarrow}
\newcommand{\ee}{\end{equation}}
\newcommand{\bea}{\begin{eqnarray}}
\newcommand{\eea}{\end{eqnarray}}
\newcommand{\bna}{\begin{eqnarray*}}
\newcommand{\ena}{\end{eqnarray*}}

\renewcommand{\le}{\left}
\newcommand{\ri}{\right}

\usepackage{float,amssymb,mathrsfs,bm,multirow,graphics,fancyhdr}

\usepackage{geometry}
\journal{***}

\begin{document}

\begin{frontmatter}

\title{Calculus of variations on hypergraphs}

\author[qnu]{Mengqiu Shao}
\ead{mqshaomath@qfnu.edu.cn}
\address[qnu]{School of Mathematical Sciences, Qufu Normal University, Shandong, 273165, China}

\author[bnu]{Yulu Tian\corref{Tian}}
\ead{tianyl@mail.bnu.edu.cn}
%\address[bnu]{School of Mathematics, Renmin University of China, Beijing, 100872, China}
\author[bnu]{Liang Zhao}
\ead{liangzhao@bnu.edu.cn}
\address[bnu]{School of Mathematical Sciences, Key Laboratory of Mathematics and Complex Systems of MOE,\\
	Beijing Normal University, Beijing, 100875, China}

\cortext[Tian]{Corresponding author.}

\begin{abstract}
  We have established a coherent framework for applying variational methods to partial differential equations on hypergraphs, which includes the propositions of calculus and function spaces on hypergraphs. Several results related to the maximum principle on hypergraphs have also been proven. As applications, we demonstrated how these can be used to study partial differential equations on hypergraphs.
\end{abstract}


\end{frontmatter}

\section{Introduction}

The hypergraph is an extension of the graph theory, where edges can connect more than two vertices, forming hyperedges. This sophisticated mathematical structure allows for the modeling of complex relationships that cannot be encapsulated by simple pairwise connections, making hypergraphs a powerful tool in various fields of science and engineering. In computational biology, they are used to model the intricate web of interactions between proteins \cite{fan2022identification, guan2022mvhrkm}, genes \cite{feng2020hypergraph}, and metabolic pathways \cite{krieger2023shortest,mithani2009rahnuma}. When social structures are represented with individuals as vertices and social ties as hyperedges, hypergraphs underpin the theoretical foundations of social network analysis \cite{yu2021selfsupervised,zheng2018novel}, which allows for a more nuanced view of social groups and collaborative activities. These applications of hypergraphs provide a versatile language for describing and analyzing complex systems that exhibit multi-way relationships. Their mathematical properties are continually being explored, leading to new insights and advancements in both theoretical and applied contexts.

Specifically, this paper attempts to define several fundamental calculus concepts on hypergraphs, including the gradient, divergence and Laplace operators. These concepts on hypergraphs have been studied by many mathematicians. For instance, a series of papers by Jost, Reff, Rusnak, etc. \cite{chen2018characterization, jost2019hypergraph, jost2021normalized, jost2022plaplace, kitouni2019lower, mulas2021spectral,  reff2014spectral, reff2012oriented} investigated various Laplace operators on oriented or chemical hypergraphs, where they primarily dedicated to exploring the relationship between the spectrum of Laplace operator and the structure of hypergraphs. The Laplace operator on hypergraphs and its spectrum have also been studied from an application standpoint,  and have been applied to machine learning problems on hypergraphs \cite{ fu2019hplapGCN, prokopchik2022nonlinear,saito2022hypergraph, saito2022generalizing, saito2018hypergraph}.

We plan to investigate partial differential equations on hypergraphs. Therefore, we adhere to the following three principles in this paper: 

\noindent$(1)$ Properties from classical calculus, such as the divergence theorem, are still applicable.

\noindent$(2)$ Significant properties in variational calculus, such as the maximum principle, remain valid. 

\noindent$(3)$ Concepts on hypergraphs should remain consistent with the concepts defined on graphs.

Compared to these existing works, the value of our work lies in establishing a reasonable analytical framework on hypergraphs, and under this framework, variational methods can be used to study several partial differential equations on hypergraphs. The structure of this paper is as follows: In Section \ref{basic}, we introduce the basic concepts and define the Laplace operator on hypergraphs. In Section \ref{principle}, we establish several maximum principles on hypergraphs. In Section \ref{pde}, we study several kinds of partial differential equations on hypergraphs using variational methods.

\section{Calculus on hypergraphs}\label{basic}

A directed hypergraph $H$ is a pair $(V,\vec{E})$, where $V$ is a finite set of vertices and $\vec{E}$ is the directed hyperedge set which is a subset of 
$
\overunderset{|V|}{k=2}\bigcup\underset{~\{v_1,\cdots,v_k\}\subset V}\bigcup\left\{\{v_{\sigma(1)},v_{\sigma(2)},\cdots,v_{\sigma(k)}\}, \sigma\in\mathcal{S}_k\right\}.$ 
Here $|V|$ denotes the cardinality of $V$ and $\mathcal{S}_k$ denotes the set of permutations of $\{1,2,\cdots,k\}$. For convenience, the directed hyperedge $\{v_{\sigma(1)},v_{\sigma(2)},\cdots,v_{\sigma(k)}\}$ is written as $\vec{e}_{\sigma(1)\sigma(2)\cdots \sigma(k)}$ and sometimes denoted by $\vec{e}$ for brevity. If a directed hypergraph $H$ satisfies that once $\vec{e}_{12\cdots k}\in \vec{E}$, then for any permutation $\sigma\in \mathcal{S}_k$, $\vec{e}_{\sigma(1)\sigma(2)\cdots\sigma(k)}$ is also in $\vec{E}$, then we call $H$ a symmetric directed hypergraph.

An undirected hypergraph is a hypergraph that does not distinguish the hyperedges $\vec{e}_{12\cdots k}$ and $\vec{e}_{\sigma(1)\sigma(2)\cdots \sigma(k)}$ and we simply denote its hyperedges as $e_{12\cdots k}$, or sometimes $e$, for brevity. There exists a natural correspondence between  a symmetric directed hypergraph and an undirected hypergraph through modulo permutation groups, that is, $E:=\vec{E}/ \mathcal{S}$, where $\mathcal{S}=\overunderset{|V|}{k=2}{\bigcup}\mathcal{S}_k$. From now on, we will not distinguish between an undirected hypergraph and its corresponding symmetric directed hypergraph. For two different vertices $v_i, v_j \in V$, if there is a sequence of hyperedges $\{e_1,e_2,\cdots,e_l\}$ such that $v_i\in e_1$, $v_j\in e_l$ and $e_n\cap e_{n+1}\neq\emptyset$ for $1\leq n\leq l-1$, then vertices $v_i$ and $v_j$ are connected by the hyperpath $\gamma=\{e_1,e_2,\cdots,e_l\}$. If each pair of two different vertices in a hypergraph $H$ can be connected by a hyperpath, we call $H$ a connected hypergraph. Going forward, unless otherwise noted, any mention of a hypergraph will refer to an undirected and connected hypergraph or its corresponding symmetric directed hypergraph. 

\begin{example}\label{exampleh4}
	Consider the following hypergraph $H_4$. The vertex set is $V=\{v_1,v_2,v_3,v_4\}$ and the hyperedge set is $E=\{e_{123},e_{14}\}$. If $H_4$ is regarded as a symmetric directed hypergraph, then its directed hyperedge set is $\vec{E}=\{\vec{e}_{123},\vec{e}_{132},\vec{e}_{213},\vec{e}_{231},\vec{e}_{312},\vec{e}_{321},\vec{e}_{14},\vec{e}_{41}\}.$ 

\begin{figure}[H]
\centering
\includegraphics[width=0.25\textwidth]{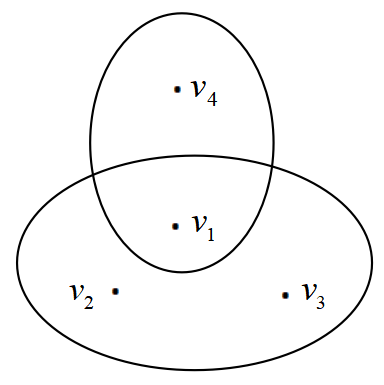}
\caption{The hypergraph $H_{4}$}
\label{figexampleh4}
\end{figure}	
\end{example}

\begin{remark}
	A hyperedge in an undirected hypergraph is typically viewed as a collection of vertices that is unordered and has no direction. However, when the context of the problem requires consideration of the orientation of hyperedges, there are several different reasonable choices. For example, in the papers by Jost and Mulas \cite{jost2019hypergraph, jost2021normalized}, with the context of chemical reactions, a chemical hypergraph is defined, which divides the vertices contained in a hyperedge into inputs (educts) and outputs (products), with the direction of the hyperedge going from inputs to outputs.
	
	In contrast, we consider a directed hyperedge as an ordered set of vertices, which is applicable when considering Markov processes of interacting particle systems on hypergraphs \cite{ kuehn2022vlasov, keliger2023accuracy, sridhar2023mean}. Moreover, our setting of ordered vertices in hyperedges can also be adopted when dealing with hypergraph cut \cite{saito2022hypergraph, saito2022generalizing, saito2018hypergraph}.
	
	A graph is a special kind of hypergraph. When we consider these different types of directed hyperedges in a graph, they are actually the same. They all become directed edges of the graph. From a mathematical perspective, we regard directed hyperedges as ordered sets of vertices, in order to ensure that the maximum principle on hypergraphs holds, which is, of course, important for analysis on hypergraphs.
\end{remark}

The measure $\mu$ on the vertex set is defined as a positive function $\mu: V\to\mathbb{R}^+$. The weight of the hyperedge set $E$ is associated with a weight function $\omega:E\to\mathbb{R}^+$. The degree of a hyperedge $e\in E$ is defined as $\delta_e:=|e|$, where $|e|$ denotes the cardinality of vertices contained in $e$, while the degree of its corresponding directed hyperedge $\vec{e}$ is denoted by $\delta_{\vec{e}}$ and then $\delta_{\vec{e}}=\delta_e=|e|$, since $e$ and $\vec{e}$ contain the same vertices.

Since there are a total of $\delta_e !$ directed hyperedges that contain the same vertices as $e$, with the order of vertices differing by a permutation $\sigma$, when calculating the degree of a vertex $v$, we assign the same weight to these directed hyperedges and multiply the weight of the hyperedge $e$ containing the vertex $v$ by $\delta_e !$. Therefore, the degree of a vertex $v\in V$ is defined as
\begin{equation*}
	d_v:=\sum_{e\in E: v\in e}\delta_e !\omega(e)=\sum_{\vec{e}\in \vec{E}: v\in\vec{e}}\omega(\vec{e}),
\end{equation*}
where $\omega(e)$ and $\omega(\vec{e})$ represent the weights of hyperedge $e\in E$ and its corresponding directed hyperedge $\vec{e}\in \vec{E}$, which are also denoted by $\omega_e$ and $\omega_{\vec{e}}$ for brevity.

The integral of a function $\phi$ over $V$ is
\begin{displaymath}
	\int_{V} \phi d\mu:=\sum_{v\in V} \mu(v)\phi(v).
\end{displaymath}
For any  $1\leq q<\infty$, we define $L^q(V)$ as the linear space of functions $\phi:V\rightarrow \mathbb{R} $ with the norm
$$
\Vert \phi\Vert_q:=\left(\int_{V}|\phi|^q d\mu\right)^{1/q}.
$$
While for $q=+\infty$, $L^\infty(V)$ is the space with the norm
$$
\Vert \phi\Vert_\infty:=\max_V |\phi|.
$$
In particular, $L^2(V)$ is a Hilbert space endowed with the inner product
\begin{equation}\label{inner1}
	\langle \phi,\psi\rangle_{V}:=\sum_{v\in V}\mu(v)\phi(v)\psi(v).
\end{equation}

For real-valued functions $\phi,\psi$ defined on the hyperedge set $E$, the Hilbert space $\mathcal{H}(E)$ is defined with the inner product
\begin{equation}\label{inner2}
	\langle \phi,\psi\rangle_{E}:=\sum_{e\in E}\phi(e)\psi(e).
\end{equation}
We can also calculate the inner product on the directed hyperedge set $\vec{E}$ corresponding to $E$,  that is,
\begin{equation}\label{inner3}
\langle \phi,\psi\rangle_{\vec{E}}:=\sum_{\vec{e}\in \vec{E}}\frac{1}{\delta_{\vec{e}} !}\phi(\vec{e})\psi(\vec{e}).
\end{equation}
The space $\mathcal{H}(\vec{E})$ defined with the inner product \eqref{inner3} is also a Hilbert space. If  $\phi(\vec{e}_{12\cdots k})=\phi(\vec{e}_{\sigma(1)\sigma(2)\cdots \sigma(k)})$ and $\psi(\vec{e}_{12\cdots k})=\psi(\vec{e}_{\sigma(1)\sigma(2)\cdots \sigma(k)})$ for any $\sigma\in \mathcal{S}_k$, then there holds $\langle \phi,\psi\rangle_{E}=\langle \phi,\psi\rangle_{\vec{E}}$. 

Next, let us define gradient and divergence operators on a hypergraph $H$.

\begin{definition}\label{grad1}
The hypergraph gradient operator $\nabla:L^2(V)\to\mathcal{H}(\vec{E})$ is defined as
\begin{equation*}
 \nabla \phi(\vec{e}):=\sqrt{\frac{\omega_{\vec{e}}}{\delta_{\vec{e}}-1}}\sum_{u\in\vec{e}}\left(\phi(u)-\phi(\vec{e}(1))\right),
\end{equation*}
where $\phi\in L^2(V)$ and $\vec{e}(1)$ is the first vertex of the directed hyperedge $\vec{e}$.
\end{definition}

\begin{remark}\label{re3.2}
	For two directed hyperedges $\vec{e}_i,\vec{e}_j\in \vec{E}$ containing the same vertices, if ${\vec{e}_i}(1)={\vec{e}_j}(1)=v$,
	then $\nabla \phi(\vec{e}_i)=\nabla \phi(\vec{e}_j)$, where ${\vec{e}_i}(1)$ and ${\vec{e}_j}(1)$ are the first vertex of these two hyperedges.
\end{remark}

\begin{remark}\label{grad2}
According to Definition \ref{grad1}, if there are totally $l_v$ hyperedges $e_1, e_2,\cdot\cdot\cdot, e_{l_v}$ such that $v\in e_i$, $i=1,2,\cdot\cdot\cdot, l_v$, then $\nabla \phi$ at any $v\in V$ is an $l_v$ dimensional vector as follows
\begin{align*}
 \nabla \phi(v)&= \left(\sqrt{\frac{\omega_{e_1}}{\mu(v)\delta_{e_1}(\delta_{e_1}-1)}}\sum_{u\in e_1}(\phi(u)-\phi(v)),
 \sqrt{\frac{\omega_{e_2}}{\mu(v)\delta_{e_2}(\delta_{e_2}-1)}}\sum_{u\in e_2}(\phi(u)-\phi(v)),\right.\\
 &\ \ \ \  \left.\cdots,
 \sqrt{\frac{\omega_{e_{l_v}}}{\mu(v)\delta_{e_{l_v}}(\delta_{e_{l_v}}-1)}}\sum_{u\in e_{l_v}}(\phi(u)-\phi(v))\right).
\end{align*}
\end{remark}

According to the Definition~\ref{grad1} and Remark~\ref{grad2}, we have the following conclusion.

\begin{lemma}\label{grad3}
For any $\phi, \psi\in L^2(V)$, there holds
\begin{equation*}
 \langle\nabla \phi,\nabla \psi\rangle_{\vec{E}}=\sum_{v\in V}\mu(v)\nabla \phi(v)\nabla\psi(v)=\int_{V}\nabla\phi\nabla\psi d\mu.
\end{equation*}
In particular, we have
\begin{equation*}
  \|\nabla\phi\|_{\vec{E}}^2=\int_{V}\vert\nabla\phi\vert^2d\mu,
\end{equation*}
where $\|\nabla\phi\|_{\vec{E}}^2:=\langle\nabla \phi,\nabla \phi\rangle_{\vec{E}}$.
\end{lemma}

\begin{proof}
By Remark~\ref{re3.2} and direct calculations, we have
	\begin{align*}
\langle\nabla\phi,\nabla\psi\rangle_{\vec{E}}
&=\sum_{\vec{e}\in\vec{E}}\frac{1}{\delta_{\vec{e}}!}\nabla\phi(\vec{e})\nabla\psi(\vec{e})\\
&=\sum_{\vec{e}\in\vec{E}}\frac{1}{\delta_{\vec{e}}!}\frac{\omega_{\vec{e}}}{(\delta_{\vec{e}}-1)}
	\left[\sum_{u\in\vec{e}}(\phi(u)-\phi(\vec{e}(1)))\right]
	\left[\sum_{u\in\vec{e}}(\psi(u)-\psi(\vec{e}(1)))\right]\\
&=\sum_{e\in E}\frac{\omega_e}{\delta_e!(\delta_e-1)}\sum_{v\in e}\left[\sum_{u\in e}
(\phi(u)-\phi(v))\right]\left[\sum_{u\in e}(\psi(u)-\psi(v))\right](\delta_e-1)!\\
&=\sum_{e\in E}\frac{\omega_e}{\delta_e(\delta_e-1)}\sum_{v\in e}\left[\sum_{u\in e}(\phi(u)-\phi(v))\right]
\left[\sum_{u\in e}(\psi(u)-\psi(v))\right]\\
&=\sum_{v\in V}\sum_{e\in E: v\in e}\frac{\omega_e}{\delta_e(\delta_e-1)}\left[\sum_{u\in e}(\phi(u)-\phi(v))\right]
\left[\sum_{u\in e}(\psi(u)-\psi(v))\right]\\
&=\sum_{v\in V}\mu(v)\frac{1}{\mu(v)}\sum_{e\in E:v\in e}\frac{\omega_e}{\delta_e(\delta_e-1)}
\left[\sum_{u\in e}(\phi(u)-\phi(v))\right]
\left[\sum_{u\in e}(\psi(u)-\psi(v))\right]\\
&=\underset{v\in V}\sum\mu(v)\nabla\phi(v)\nabla\psi(v)\\
&=\int_{V}\nabla\phi\nabla\psi d\mu.
\end{align*}
\end{proof}

Just as in Euclidean space and on manifolds, the divergence on a hypergraph is the adjoint operator of the gradient, and from this perspective, it is uniquely determined.

\begin{lemma}\label{div}
The divergence operator $\emph{\text{div}}$ on the hypergraph maps from $\mathcal{H}(\vec{E})$ to $L^2(V)$ and $\forall\psi\in L^2(V)$, $\forall\phi\in\mathcal{H}(\vec{E})$, it satisfies
\begin{equation}\label{div1}
	\langle\nabla\psi,\phi\rangle_{\vec{E}}=\langle\psi,-\emph{\text{div}}\phi\rangle_{V}.
\end{equation}
Moreover, according to \eqref{div1}, the divergence operator is uniquely determined by the following equation
\begin{equation}\label{div2}
	\emph{\text{div}}\phi(v)=-\frac{1}{\mu(v)}\sum_{\vec{e}\in \vec{E}: v\in \vec{e}}\frac{\sqrt{\omega_{\vec{e}}}}{\delta_{\vec{e}}!\sqrt{\delta_{\vec{e}}-1}}\phi(\vec{e})+\frac{1}{\mu(v)}\sum_{\vec{e}\in \vec{E}: \vec{e}(1)=v}\frac{\delta_{\vec{e}}\sqrt{\omega_{\vec{e}}}}{\delta_{\vec{e}}!\sqrt{\delta_{\vec{e}}-1}}\phi(\vec{e}).
\end{equation}
\end{lemma}

\begin{proof} 
Take
\begin{equation*}
	\psi_v(u)=\left\{
	\begin{aligned}
		&1, &&\text{if } u=v\in V,\\
		&0, &&\text{otherwise}.
	\end{aligned}
    \right.
\end{equation*}
By the definition of inner product in $\mathcal{H}(\vec{E})$, we have
\begin{align}\label{Div}
\langle \nabla\psi_v,\phi\rangle_{\vec{E}}
&=\sum_{\vec{e}\in \vec{E}}\frac{1}{\delta_{\vec{e}}!}\nabla\psi_v(\vec{e})\phi(\vec{e})\nonumber\\
&=\sum_{\vec{e}\in \vec{E}}\frac{1}{\delta_{\vec{e}}!}\left(\frac{\sqrt{\omega_{\vec{e}}}}{\sqrt{\delta_{\vec{e}}-1}}
  \sum_{u\in\vec{e}}(\psi_{v}(u)-\psi_{v}(\vec{e}(1)))\right)\phi(\vec{e})\nonumber\\
&=\sum_{\vec{e}\in \vec{E}: v\in \vec{e}}\frac{1}{\delta_{\vec{e}}!}
  \left(\frac{\sqrt{\omega_{\vec{e}}}}{\sqrt{\delta_{\vec{e}}-1}}\sum_{u\in\vec{e}}\psi_{v}(u)
  \right)\phi(\vec{e})-\sum_{\vec{e}\in E: v\in \vec{e}}\frac{1}{\delta_{\vec{e}}!}\frac{\sqrt{\omega_{\vec{e}}}}
  {\sqrt{\delta_{\vec{e}}-1}}\delta_{\vec{e}}~\psi_{v}(\vec{e}(1))\phi(\vec{e})\nonumber\\
&=\sum_{\vec{e}\in \vec{E}: v\in \vec{e}}\frac{\sqrt{\omega_{\vec{e}}}}{\delta_{\vec{e}}!\sqrt{\delta_{\vec{e}}-1}}\phi(\vec{e})-\sum_{\vec{e}\in \vec{E}: \vec{e}(1)=v}\frac{\delta_{\vec{e}}\sqrt{\omega_{\vec{e}}}}{\delta_{\vec{e}}!\sqrt{\delta_{\vec{e}}-1}}\phi(\vec{e}).
\end{align}
On the other hand, it follows from \eqref{div1} that
\begin{equation}\label{Div1} 
\langle\nabla\psi_v,\phi\rangle_{\vec{E}}=
\langle\psi_v,-\text{div}\phi\rangle_{V}=-\mu(v)\text{div}\phi(v).
\end{equation}
Therefore, \eqref{div2} is obtained by \eqref{Div} and \eqref{Div1}.
\end{proof}

Consequently, the Laplace operator on hypergraph $\Delta:L^2(V)\to L^2(V)$ shall be defined by
\begin{equation}\label{laplace}
	\Delta\phi:=\text{div}(\nabla\phi),
\end{equation} 
and \eqref{div1} tells us that for any $\phi,\psi\in L^2(V)$, there holds
\begin{equation}\label{laplace1}
	\langle \nabla\psi,\nabla\phi\rangle_{\vec{E}}=\langle\psi,-\text{div}(\nabla\phi)\rangle_{V}=\langle\psi,-\Delta\phi\rangle_{V}.
\end{equation}
Moreover, we can compute the Laplacian of a function at a vertex $v\in V$ as follows.
\begin{lemma}\label{laplacian}
For any $\phi\in L^2(V)$, there holds
\begin{align}\label{h-laplace}
\Delta\phi(v)&=-\frac{1}{\mu(v)}\sum_{e\in E:v\in e}\left[\omega_e\phi(v)-\sum_{u\in e:u\neq v}\frac{\omega_e}{\delta_e-1}\phi(u)\right]\nonumber\\
&=\frac{1}{\mu(v)}\sum_{e\in E:v\in e}\frac{\omega_e}{\delta_e-1}\sum_{u\in e:u\neq v}(\phi(u)-\phi(v)).
\end{align}
\end{lemma}

\begin{proof}
 Substitutting \eqref{div2} into \eqref{laplace} and by Remark \ref{re3.2}, we obtain
\begin{align*}
	\Delta\phi(v)&=\text{div}(\nabla\phi)(v)\\
&=-\frac{1}{\mu(v)}\sum_{\vec{e}\in \vec{E}: v\in \vec{e}}\frac{\sqrt{\omega_{\vec{e}}}}{\delta_{\vec{e}}!\sqrt{\delta_{\vec{e}}-1}}\nabla\phi(\vec{e})+\frac{1}{\mu(v)}\sum_{\vec{e}\in \vec{E}: \vec{e}(1)=v}\frac{\delta_{\vec{e}}\sqrt{\omega_{\vec{e}}}}{\delta_{\vec{e}}!\sqrt{\delta_{\vec{e}}-1}}\nabla\phi(\vec{e})\\
&=-\frac{1}{\mu(v)}\sum_{e\in E:v\in e}\left[\sum_{u\in e}\left(\frac{\omega_e}{\delta_e(\delta_e-1)}
   \sum_{w\in e}(\phi(w)-\phi(u))\right) -\frac{\omega_e}{\delta_e-1}\sum_{w\in e}
   (\phi(w)-\phi(v))\right]\\
&=-\frac{1}{\mu(v)}\sum_{e\in E:v\in e}\left[\sum_{u\in e:u\neq v}
   \left(\frac{\omega_e}{\delta_e(\delta_e-1)}\sum_{w\in e}(\phi(w)-\phi(u))\right)
   +\frac{1-\delta_e}{\delta_e}\frac{\omega_e}{\delta_e-1}
   \sum_{w\in e}(\phi(w)-\phi(v))\right]\\	
&=-\frac{1}{\mu(v)}\sum_{e\in E:v\in e}\left[\sum_{u\in e:u\neq v}\frac{\omega_e}{\delta_e(\delta_e-1)}\sum_{w\in e}\phi(w)
  -\sum_{u\in e:u\neq v}\frac{\omega_e}{\delta_e(\delta_e-1)}\delta_e\phi(u)
  -\frac{\omega_e}{\delta_e}\sum_{w\in e}\phi(w)+\frac{\omega_e}
   {\delta_e}\delta_e\phi(v)\right]\\		
&=-\frac{1}{\mu(v)}\sum_{e\in E:v\in e}\left[\omega_e\phi(v)-\sum_{u\in e:u\neq v}\frac{\omega_e}{\delta_e-1}\phi(u)\right]-\frac{1}{\mu(v)}I_0\\
&=-\frac{1}{\mu(v)}\sum_{e\in E:v\in e}\left[\omega_e\phi(v)-\sum_{u\in e:u\neq v}\frac{\omega_e}{\delta_e-1}\phi(u)\right]\\
&=\frac{1}{\mu(v)}\sum_{e\in E:v\in e}\frac{\omega_e}{\delta_e-1}\sum_{u\in e:u\neq v}(\phi(u)-\phi(v)),
\end{align*}
where
\begin{equation*}
  I_0=\sum_{e\in E:v\in e}\left[\sum_{u\in e:u\neq v}\frac{\omega_e}{\delta_e(\delta_e-1)}\sum_{w\in e}\phi(w)
  -\frac{\omega_e}{\delta_e}\sum_{w\in e}\phi(w)\right]=0.
\end{equation*}
This completes the proof of the Lemma.
\end{proof}

Although the formula of integral by parts on hypergraphs can be confirmed by combining Lemma \ref{grad3}, Lemma \ref{div} and \eqref{laplace}, in order to understand the relationship between gradient and Laplace operators on hypergraphs more directly, we will provide a proof of this formula through direct calculations.

\begin{lemma}\label{fbjf}
For any $\phi,\psi\in L^2(V)$, there holds
\begin{equation*}
  \int_{V}\nabla\phi\nabla\psi d\mu=\int_{V}(-\Delta\phi)\psi d\mu.
\end{equation*}
\end{lemma}

\begin{proof}
By Remark \ref{grad2} and \eqref{h-laplace}, we have
\begin{equation*}
\begin{aligned}
\int_{V}\nabla\phi\nabla\psi d\mu
&=\sum_{v\in V}\sum_{e\in E:v\in e}\frac{\omega_e}{\delta_e(\delta_e-1)}\left[\sum_{u\in e}(\phi(u)-\phi(v))\right]
\left[\sum_{u\in e}(\psi(u)-\psi(v))\right]\\
&=\sum_{v\in V}\sum_{e\in E:v\in e}\frac{\omega_e}{\delta_e(\delta_e-1)}\left[\sum_{u\in e}(\phi(u)-\phi(v))\right]
\left[\sum_{u\in e}\psi(u)-\delta_e\psi(v)\right]\\
&=-\sum_{v\in V}\sum_{e\in E:v\in e}\frac{\omega_e}{\delta_e-1}\left[\sum_{u\in e}(\phi(u)-\phi(v))\right]\psi(v)\\
& \ \ \ \ +\sum_{v\in V}\sum_{e\in E:v\in e}\frac{\omega_e}{\delta_e(\delta_e-1)}\left[\sum_{u\in e}(\phi(u)-\phi(v))\right]\sum_{u\in e}\psi(u)\\
&=\int_{V}(-\Delta\phi)\psi d\mu+I,
\end{aligned}
\end{equation*}
where
\begin{equation}\label{I}
  I=\sum_{v\in V}\sum_{e\in E:v\in e}\frac{\omega_e}{\delta_e(\delta_e-1)}\left[\sum_{u\in e}(\phi(u)-\phi(v))\right]\sum_{u\in e}\psi(u)=0.
\end{equation}
Thus the lemma is proved.
\end{proof}

Let $\Omega$ be a non-empty finite subset of $V$ such that $\Omega^{c}$ is non-empty. We define the boundary of $\Omega$ by
\begin{equation*}
	\partial\Omega:=\left\{u\in V\setminus\Omega:v\in\Omega \text{ and } \exists\ \ e\in E \text{ such that }v,u\in e \right\}.
\end{equation*}
Let $W^{1,2}_{0}(\Omega)$ be the completion of $C_{c}(\Omega)$ under the norm
\begin{equation}\label{w012norm}
\|\phi\|_{W_0^{1,2}(\Omega)}=\left(\int_{\Omega\cup\partial\Omega}|\nabla \phi|^{2}d\mu\right)^{\frac{1}{2}},
\end{equation}
where $C_{c}(\Omega)$ is a set of all functions $\phi: \Omega\cup\partial\Omega\rightarrow\mathbb{R}$ satisfying $\text{supp}~\phi\subset\Omega$ and $\phi=0$ on $\partial\Omega$. Actually,  $W^{1,2}_{0}(\Omega)$ is a finite dimensional linear space since the $\Omega$ only contains finite vertexes and the norm \eqref{w012norm} is equivalent to the classical norm $\left(\int_{\Omega\cup\partial\Omega}|\nabla \phi|^{2}d\mu+\int_{\Omega}\phi^2d\mu\right)^{\frac{1}{2}}$ on $W^{1,2}_{0}(\Omega)$. In particular, we obtain the following formula of integral by parts.
\begin{lemma}
For any $\psi\in C_{c}(\Omega)$, there holds
\begin{equation}\label{fbjf0}
  \int_{\Omega\cup\partial\Omega}\nabla\phi\nabla\psi d\mu=\int_{\Omega}(-\Delta\phi)\psi d\mu,\ \ \forall \phi\in W^{1,2}_0(\Omega) .
\end{equation}
\end{lemma}

\begin{proof}
By Remark \ref{grad2} and straightforward calculation, we get
\begin{align}\label{fbjf2}
\int_{\Omega}\nabla\phi\nabla\psi d\mu
&=\sum_{v\in \Omega}\sum_{e\in E:v\in e}\frac{\omega_e}{\delta_e(\delta_e-1)}\left[\sum_{u\in e}(\phi(u)-\phi(v))\right]
\left[\sum_{u\in e}(\psi(u)-\psi(v))\right]\nonumber\\
&=\sum_{v\in \Omega}\sum_{e\in E:v\in e}\frac{\omega_e}{\delta_e(\delta_e-1)}\left[\sum_{u\in e}(\phi(u)-\phi(v))\right]
\left[\sum_{u\in e}\psi(u)-\delta_e\psi(v)\right]\nonumber\\
&=-\sum_{v\in \Omega}\sum_{e\in E:v\in e}\frac{\omega_e}{\delta_e-1}\left[\sum_{u\in e}(\phi(u)-\phi(v))\right]\psi(v)\nonumber\\
&\ \ \ \ +\sum_{v\in \Omega}\sum_{e\in E:v\in e}\frac{\omega_e}{\delta_e(\delta_e-1)}\left[\sum_{u\in e}(\phi(u)-\phi(v))\right]\sum_{u\in e}\psi(u)\nonumber\\
&=\int_{\Omega}(-\Delta \phi)\psi d\mu+II,
\end{align}
where
\begin{equation*}
  II=\sum_{v\in \Omega}\sum_{e\in E:v\in e}\frac{\omega_e}{\delta_e(\delta_e-1)}\left[\sum_{u\in e}(\phi(u)-\phi(v))\right]\sum_{u\in e}\psi(u).
\end{equation*}
Note that, for any $\psi\in C_{c}(\Omega)$, $\psi$ is naturally viewed as a  function defined on $V$, say $\psi\equiv 0$ on $\Omega^{c}$. 
Then by \eqref{I}, we have
\begin{equation*}
\begin{aligned}
I &=\underset{v\in V}\sum\underset{{e\in E:v\in e}}\sum\frac{\omega_e}{\delta_e(\delta_e-1)}\left[\sum_{u\in e}(\phi(u)-\phi(v))\right]\sum_{u\in e}\psi(u)\\
    &= \underset{v\in \Omega}\sum\underset{{e\in E:v\in e}}\sum\frac{\omega_e}{\delta_e(\delta_e-1)}\left[\sum_{u\in e}(\phi(u)-\phi(v))\right]\sum_{u\in e}\psi(u)\\
    &\ \ \ \ +\underset{v\in \Omega^c}\sum\underset{{e\in E:v\in e}}\sum\frac{\omega_e}{\delta_e(\delta_e-1)}\left[\sum_{u\in e}(\phi(u)-\phi(v))\right]\left[\sum_{u\in e}(\psi(u)-\psi(v))\right]\\
    %&=II+\underset{v\in \Omega^c}\sum\underset{{e\in E:v\in e}}\sum\frac{\omega_e}{\delta_e(\delta_e-1)}\left[\sum_{u\in e}(\phi(u)-\phi(v))\right]\left[\sum_{u\in e}(\psi(u)-\psi(v))\right]\\
    &=II+\int_{\Omega^c}\nabla\phi\nabla\psi d\mu,
\end{aligned}
\end{equation*}
which implies that
\begin{equation}\label{II}
  II=I-\int_{\Omega^c}\nabla\phi\nabla\psi d\mu=\int_{\Omega^c}\nabla\phi\nabla\psi d\mu.
\end{equation}
Next, we claim that
\begin{equation}\label{IIa}
  \int_{\Omega^c}\nabla\phi\nabla\psi d\mu=\int_{\partial\Omega}\nabla\phi\nabla\psi d\mu, \ \ \forall \phi\in W^{1,2}_{0}(\Omega),\ \ \forall \psi\in C_c(\Omega).
\end{equation}
Certainly, if $\Omega^c=\partial\Omega$, then \eqref{IIa}  holds evidently. Next, we prove the case of $\partial\Omega\subset\Omega$ but $\partial\Omega\neq\Omega$. Since for any $v\in\Omega^c\setminus\partial\Omega$ and $u,v\in e$, we have $u\in\Omega^c$ and then
\begin{align}\label{omegac}
  \int_{\Omega^c}\nabla\phi\nabla\psi d\mu
   &=\underset{v\in \Omega^c}\sum\underset{{~e\in E:v\in e}}\sum\frac{\omega_e}{\delta_e(\delta_e-1)}
      \left[\sum_{u\in e}(\phi(u)-\phi(v))\right]\left[\sum_{u\in e}(\psi(u)-\psi(v))\right]\nonumber\\
   &=\underset{v\in \Omega^c}\sum\underset{{~e\in E:v\in e}}\sum\frac{\omega_e}{\delta_e(\delta_e-1)}
      \sum_{u\in e}\phi(u)\cdot\sum_{u\in e}\psi(u)\nonumber\\
   &=\underset{v\in \partial\Omega}\sum\underset{{~e\in E:v\in e}}\sum\frac{\omega_e}{\delta_e(\delta_e-1)}
      \sum_{u\in e}\phi(u)\cdot\sum_{u\in e}\psi(u)\nonumber\\  
   &\ \ \ \  +\underset{v\in \Omega^c\setminus\partial\Omega}\sum\underset{{~e\in E:v\in e}}\sum\frac{\omega_e}{\delta_e(\delta_e-1)}
      \sum_{u\in e, u\in\Omega^c}\phi(u)\cdot\sum_{u\in e, u\in\Omega^c}\psi(u).
\end{align}
Thus by \eqref{omegac}, for any $\phi\in W^{1,2}_{0}(\Omega)$ and $\psi\in C_c(\Omega)$, we obtain
\begin{equation*}
  \int_{\Omega^c}\nabla\phi\nabla\psi d\mu=\int_{\partial\Omega}\nabla\phi\nabla\psi d\mu+III, 
\end{equation*}
where
\begin{equation*}
  III=\underset{v\in \Omega^c\setminus\partial\Omega}\sum\underset{{~e\in E:v\in e}}\sum\frac{\omega_e}{\delta_e(\delta_e-1)}
      \sum_{u\in e, u\in\Omega^c}\phi(u)\cdot\sum_{u\in e, u\in\Omega^c}\psi(u)=0,
\end{equation*}
which implies \eqref{IIa}.

Therefore, it follows from \eqref{fbjf2}, \eqref{II} and \eqref{IIa} that \eqref{fbjf0} holds. 
\end{proof}

\section{Maximum Principles on hypergraphs}\label{principle}

In this section, we introduce several maximum principle on hypergraphs.

\begin{lemma}\label{Wmp}
\emph{(Weak maximum principle)} Assume that the function $c(v)>0$ for any $v\in V$. If the function $\phi$ defined on $V$ satisfies $-\Delta \phi+c(v)\phi\geq0$,
then $\phi\geq 0$ on $V$.
\end{lemma}

\begin{proof}
Let $\phi^-=\min\{\phi,0\}$. For any $v\in V$, we claim that
\begin{equation}\label{weak}
 -\Delta \phi^-(v)+c(v)\phi^-(v)\geq0,
\end{equation}
from which, one has
\begin{align*}
  0 &\geq \langle-\Delta \phi^-+c(v)\phi^-,\phi^-\rangle_{V} \\
   &=\langle\nabla\phi^-,\nabla\phi^-\rangle_{\vec{E}}+\langle c(v)\phi^-,\phi^-\rangle_{V}\\
   &= \|\nabla\phi^-\|^2_{\vec{E}}+\underset{v\in V}\sum~\mu(v)c(v)|\phi^-(v)|^2\\
   &\geq 0.
\end{align*}
This lead to $\phi^-\equiv 0$ on $V$.
Next, we prove \eqref{weak} in the following two cases.

(i) If $\phi(v)\geq 0$, then $\phi^-(v)=0$ and thus
\begin{align*}
  -\Delta \phi^-(v)&=\frac{1}{\mu(v)}\sum_{e\in E:v\in e}\Bigg[\omega_e\phi^-(v)-\sum_{u\in e:u\neq v}\frac{\omega_e}{\delta_e-1}\phi^-(u)\Bigg]\\
  &=-\frac{1}{\mu(v)}\sum_{e\in E:v\in e}\sum_{u\in e:u\neq v}\frac{\omega_e}{\delta_e-1}\phi^-(u)\\
  &\geq0,
\end{align*}
since $\phi^-(z)\leq 0$ for any $z\in V$. 
Therefore $-\Delta \phi^-(v)+c(v)\phi^-(v)=-\Delta \phi^-(v)\geq0$.

(ii) If $\phi(v)<0$, one has $\phi^-(v)=\phi(v)$ and thus
\begin{align*}
  -\Delta \phi^-(v)&=\frac{1}{\mu(v)}\sum_{e\in E:v\in e}\Bigg[\omega_e\phi^-(v)-\sum_{u\in e:u\neq v}\frac{\omega_e}{\delta_e-1}\phi^-(u)\Bigg]\\
   &= \frac{1}{\mu(v)}\sum_{e\in E:v\in e}\Bigg[\omega_e\phi(v)-\sum_{u\in e:u\neq v}\frac{\omega_e}{\delta_e-1}\phi^-(u)\Bigg]\\
   &\geq\frac{1}{\mu(v)}\sum_{e\in E:v\in e}\Bigg[\omega_e\phi(v)-\sum_{u\in e:u\neq v}\frac{\omega_e}{\delta_e-1}\phi(u)\Bigg]\\
   &=-\Delta \phi(v),
\end{align*}
which implies that $-\Delta \phi^-(v)+c(v)\phi^-(v)\geq-\Delta \phi(v)+c(v)\phi(v)\geq0$.

Combining (i) and (ii), we get \eqref{weak}. The proof of this lemma is completed.
\end{proof}

If $c\equiv0$ in Lemma~\ref{Wmp}, then we have the following weak maximum principle.

\begin{lemma}\label{lem2.1} If the function $\phi$ defined on $V$ satisfies
	$-\Delta\phi\geq0$, then $\phi^-$ is a constant function, where $\phi^-=\min\{\phi,0\}$.
\end{lemma}

\begin{proof}
	For any $v\in V$, we claim that
	\begin{equation}\label{weak1}
		-\Delta\phi^-(v)\geq0.
	\end{equation}
Indeed, if $\phi(v)\geq0$, then $\phi^-(v)=0$ and 
\begin{align*}
-\Delta\phi^-(v)&=\frac{1}{\mu(v)}\sum_{e\in E:v\in e}\left[\omega_e\phi^-(v)
                  -\sum_{u\in e:u\neq v}\frac{\omega_e}{(\delta_e-1)}\phi^-(u)\right]\\
			    &=-\frac{1}{\mu(v)}\sum_{e\in E:v\in e}\sum_{u\in e:u\neq v}\frac{\omega_e}
                  {(\delta_e-1)}\phi^-(u)\\
                &\geq0,
\end{align*}
where the last inequality is due to $\phi^-(z)\leq0$ for any $z\in V$. If $\phi(v)<0$, then we get
\begin{align*}
-\Delta\phi^-(v)&=\frac{1}{\mu(v)}\sum_{e\in E:v\in e}\left[\omega_e\phi^-(v)
                  -\sum_{u\in e:u\neq v}\frac{\omega_e}{(\delta_e-1)}\phi^-(u)\right]\\
			    &=\frac{1}{\mu(v)}\sum_{e\in E:v\in e}\left[\omega_e\phi(v)
                  -\sum_{u\in e:u\neq v}\frac{\omega_e}{(\delta_e-1)}\phi^-(u)\right]\\
			    &\geq\frac{1}{\mu(v)}\sum_{e\in E:v\in e}\left[\omega_e\phi(v)
                  -\sum_{u\in e:u\neq v}\frac{\omega_e}{(\delta_e-1)}\phi(u)\right]\\
			    &=-\Delta\phi(v)\\
                &\geq0.
		\end{align*}
Therefore, $-\Delta\phi^-(v)\geq0$ for any $v\in V$, which confirms \eqref{weak1}.

On the other hand, by \eqref{laplace1} we have
\begin{equation*}	
0\geq\langle-\Delta\phi^-,\phi^-\rangle_{V}=\langle\nabla\phi^-,\nabla\phi^-\rangle_{\vec{E}}=\Vert\nabla\phi^-\Vert^2_{\vec{E}}\geq0,
\end{equation*}
and then
\begin{equation*}
\nabla\phi^-(\vec{e})=0,\ \ \forall \vec{e}\in \vec{E}.
\end{equation*}
Note that
\begin{equation*}
\nabla\phi^-(\vec{e})=\frac{\sqrt{\omega_{\vec{e}}}}{\sqrt{\delta_{\vec{e}}-1}}\sum^{\delta_{\vec{e}}}_{i=1}(\phi^-_i-\phi^-_1)=0,
\end{equation*}
where $\phi^-_i$ is the value of $\phi^-$ at the $i$-th vertex of the directed hyperedge $\vec{e}$.
Then
\begin{equation}\label{phi}
  \sum^{\delta_{\vec{e}}}_{i=1}(\phi^-_i-\phi^-_1)=0.
\end{equation}
By the symmetry of the hypergraph, we know that $H$ contains all directed hyperedges obtained by permuting the vertices within $\vec{e}$.
Then we have
\begin{equation*}
\sum^{\delta_{\vec{e}}}_{i=1}(\phi^-_i-\phi^-_1)=0, \sum^{\delta_{\vec{e}}}_{i=1}(\phi^-_i-\phi^-_2)=0, \cdots, \sum^{\delta_{\vec{e}}}_{i=1}(\phi^-_i-\phi^-_{\delta_{\vec{e}}})=0
	\end{equation*}
and 
	\begin{equation*}
		\phi^-_1=\phi^-_2=\cdots=\phi^-_{\delta_{\vec{e}}}=\frac{1}{\delta_{\vec{e}}}\sum^{\delta_{\vec{e}}}_{i=1}\phi^-_i,
	\end{equation*}
which implies that $\phi^-$ is a constant function, since the hypergraph $H$ is connect.
\end{proof}

\begin{lemma}\label{constant}
  Let $\phi$ be a function defined on $V$.  Then  $-\Delta\phi=0$ if and only if $\phi$ is a constant function on $V$.
\end{lemma}

\begin{proof}
Obviously, if $\phi$ is a constant function on $V$, then $-\Delta\phi=0$ by using Lemma \ref{laplacian}. On the other hand,
if $-\Delta\phi=0$, then we have 
\begin{equation*} 
0=\langle-\Delta\phi,\phi\rangle_{V}=\langle\nabla\phi,\nabla\phi\rangle_{\vec{E}}=\Vert\nabla\phi\Vert^2_{\vec{E}}\geq0.
\end{equation*}
Thus
	\begin{equation*}
		\nabla\phi(\vec{e})=0,\ \ \forall \vec{e}\in \vec{E}.
	\end{equation*}
Namely,
	\begin{equation*}
		\nabla\phi(\vec{e})=\frac{\sqrt{\omega_{\vec{e}}}}{\sqrt{\delta_{\vec{e}}-1}}\sum^{\delta_{\vec{e}}}_{i=1}(\phi_i-\phi_1)=0,
	\end{equation*}
where $\phi_i$ is the values of $\phi$ at the $i$-th vertex of the directed hyperedge $\vec{e}$. Then we have 
$$\sum^{\delta_{\vec{e}}}_{i=1}(\phi_i-\phi_1)=0.$$ 
Similar to the argument of \eqref{phi}, we obtain $\phi_i=\phi_1$ for all $v_i\in \vec{e}$.  Then by the connectedness of the hypergraph $H$, we obtain that $\phi$ is a constant function.
\end{proof}

%\begin{remark}
%\textcolor[rgb]{1.00,0.00,0.00}{Since $H=(V,E)$ is a connected finite hypergraph, any function $\phi$ defined on $V$ is a bounded function. Similar to the Euclidean space, the above lemma can be regarded as the Liouville  theorem in the discrete case.} 
%\end{remark}

\begin{lemma}\label{Strong1}
\emph{(Strong maximum principle)} Assume that  $\phi\geq 0$ and $-\Delta \phi+c(v)\phi\geq0$ for some function $c(v)\geq0$, $\forall v\in V$. If  there exists $v_{0}\in V$ such that $\phi(v_0)=0$, then $\phi\equiv 0$ on $V$.
\end{lemma}

\begin{proof}
Let $v=v_0$. Then we get
\begin{equation*}
-\Delta \phi(v_0)+c(v_0)\phi(v_0)= \frac{1}{\mu(v_0)}\sum_{e\in E:v_0\in e}\Bigg[\omega_e\phi(v_0)-\sum_{u\in e:u\neq v_0}\frac{\omega_e}{\delta_e-1}\phi(u)\Bigg]+c(v_0)\phi(v_0)\geq 0,
\end{equation*}
which implies
\begin{equation}\label{strong2}
  \frac{1}{\mu(v_0)}\sum_{e\in E:v_0\in e}\sum_{u\in e:u\neq v_0}\frac{\omega_e}{\delta_e-1}\phi(u)\leq 0.
\end{equation}
Since $\phi\geq0$, $\omega_e>0$ and $\delta_e-1>0$, it follows from \eqref{strong2} that
\begin{equation*}
 \phi(u)=0,\ \ \forall u\in e,
\end{equation*}
where $v_{0}\in e$ and $u\neq v_0$. Therefore, $\phi\equiv 0$ on $V$ by the connectedness of $H$.
\end{proof}

\begin{remark}
Let $-\Delta\phi\geq0$ and $\phi(v_0)=c$ for some constant $c$. There are some interesting results as follows. By Lemma \ref{lem2.1}, we know that $\phi^-$ is a constant function. Thus, 
\begin{itemize}
  \item if $c\leq0$, then $\phi^-\equiv c$ and it follows from Lemma \ref{constant} and Lemma \ref{Strong1} that $\phi\equiv c$; 
  \item if $c>0$, then $\phi^-\equiv0$ and it follows from Lemma \ref{Strong1} that  $\phi>0$.
\end{itemize}
\end{remark}

When considering the Dirichlet problem on the connected hypergraph $H$, such as problem \eqref{dirichlet} in the subsection 4.1, we will use the following maximum principle.

\begin{lemma}\label{maxmin}
Let $\Omega$ be a non-empty finite subset of $V$ such that $\Omega^{c}$ is non-empty. If for any $\phi: V\rightarrow \mathbb{R}$, $\Delta\phi(v)\geq 0$ for all $v\in\Omega$, then
\begin{equation}\label{max}
  \underset{\Omega}\max\phi\leq\underset{\Omega^c}\max\phi.
\end{equation}
 If for any $\phi: V\rightarrow \mathbb{R}$, $\Delta\phi(v)\leq 0$ for all $v\in\Omega$, then
\begin{equation}\label{min}
  \underset{\Omega}\min\phi\geq\underset{\Omega^c}\min\phi.
\end{equation}
\end{lemma}

\begin{proof}
It suffices to prove \eqref{max}, since the proof of \eqref{min} is similar. Since $V$ is finite, $\underset{\Omega^c}\max\phi<+\infty$. Then, by substituting  $\phi$ with $\phi+constant$,  we can assume that $\underset{\Omega^c}\max\phi=0$.
Let
\begin{equation*}
  M=\underset{\Omega}\max\phi.
\end{equation*}
Next, we prove that $M\leq0$. By the contrary, we assume that $M>0$. Define
\begin{equation*}
  S:=\{v\in \Omega:\phi(v)=M\}.
\end{equation*}
Obviously, $S\subset \Omega$ and $S\neq\emptyset$.
\vskip4pt
\noindent {\bf Claim 1.} \emph{If $v\in S$, then for all $u\in V$ satisfying $u,v\in e$, we have $u\in S$.}
\vskip4pt
In fact, for any $v\in S$, we have
\begin{equation*}
  \Delta\phi(v)=-\frac{1}{\mu(v)}\sum_{e\in E:v\in e}\Bigg[\omega_e\phi(v)-\sum_{u\in e:u\neq v}\frac{\omega_e}{\delta_e-1}\phi(u)\Bigg]\geq 0.
\end{equation*}
Thus
\begin{equation*}
 \sum_{e\in E:v\in e}\omega_e\phi(v)\leq\sum_{e\in E:v\in e}\sum_{u\in e:u\neq v}\frac{\omega_e}{\delta_e-1}\phi(u).
\end{equation*}
Since $\phi(u)\leq M $ for all $u\in V$ satisfying $u,v\in e$, we get
\begin{equation*}
  \sum_{e\in E:v\in e}\omega_eM\leq\sum_{e\in E:v\in e}\sum_{u\in e:u\neq v}\frac{\omega_e}{\delta_e-1}\phi(u)\leq\sum_{e\in E:v\in e}\omega_eM.
\end{equation*}
Then we obtain
\begin{equation*}
  \sum_{e\in E:v\in e}\sum_{u\in e:u\neq v}\frac{\omega_e}{\delta_e-1}(M-\phi(u))=0,
\end{equation*}
which implies that $\phi(u)=M$ for all $u\in V$ satisfying $u,v\in e$, since $M-\phi(u)\geq 0$ and $\frac{\omega_e}{\delta_e-1}\geq0$.

\vskip4pt
\noindent {\bf Claim 2.} \emph{Let $A\subset V$ and $A\neq\emptyset$ such that $v\in A$ implies that $u\in A$, where $u,v\in e$ and $e\in E$. Then $A=V$.}
\vskip4pt
Indeed, let $v\in A$ and $u$ be any other vertex in $V$. Since $H$ is connected, there is a hyperpath  $\{e_1,e_2,\cdot\cdot\cdot,e_k\}$ such that $v\in e_1$, $u\in e_k$ and $e_i\cap e_{i+1}\neq\emptyset$ for all $1\leq i\leq k-1$. Note that $v\in e_1$ and $v\in A$ implies $u_{1j}\in A$, where $u_{1j}\in e_{1}, j=1,2,\cdot\cdot\cdot,\delta_{e_{1}}$. Similarly, we obtain $u_{ij}\in A$, $i=1,2,\cdot\cdot\cdot,k$, $j=1,2,\cdot\cdot\cdot,\delta_{e_{i}}$, whence $u\in A$. Thus $A=V$.

It follows from the two claims that $S=V$, which is not possible since $\phi(v)\leq 0$ in $\Omega^c$. This contradiction shows that $M\leq 0$.
\end{proof}

\section{Partial differential equations on hypergraphs}\label{pde}

Let $H$ be a connected finite undirected hypergraph or its corresponding symmetric directed hypergraph. For brevity, we call $H$ a connected finite hypergraph in the following. In this section, we investigate several classes of partial differential equations on $H$ by using variational methods.

\subsection{Linear Schr\"{o}dinger equation}

Let $H=(V,E)$ be a connected finite hypergraph. In this subsection, we study the existence and uniqueness of solutions to the following linear Schr\"{o}dinger equation
\begin{align}\label{dirichlet}
\begin{cases}
-\Delta \phi=f, &\text{in}\ \ \Omega,\\
\phi=0, &\text{on}\ \ \partial \Omega,
\end{cases}
\end{align}
where $\Omega$ is a non-empty finite subset of $V$ such that $\Omega^{c}$ is non-empty and $f:\Omega\rightarrow \mathbb{R}$ is a function. 

To study the Dirichlet problem \eqref{dirichlet}, it is natural to consider the function space $W^{1,2}_{0}(\Omega)$, 
which is a Hilbert space with its inner product
$$\langle\phi,\psi\rangle_{W^{1,2}_{0}(\Omega)}=\int_{\Omega\cup\partial\Omega}\nabla\phi\nabla\psi d\mu,\ \ \forall \phi, \psi\in W^{1,2}_{0}(\Omega).$$
The functional $\mathcal{I}: W^{1,2}_{0}(\Omega)\rightarrow\mathbb{R}$ related to \eqref{dirichlet} is defined by
\begin{equation*}
  \mathcal{I}(\phi)=\frac{1}{2}\int_{\Omega\cup\partial\Omega}|\nabla\phi|^2d\mu-\int_{\Omega}f\phi d\mu.
\end{equation*}
If for any $\psi\in C_{c}(\Omega)$, there holds
\begin{equation*}\label{weaksolutiondirichlet}
\int_{\Omega\cup\partial\Omega}\nabla \phi \nabla \psi d\mu=\int_{\Omega}f\psi d\mu, \ \ \phi\in W^{1,2}_{0}(\Omega),
\end{equation*}
then $\phi$ is called a weak solution of $(\ref{dirichlet})$. Clearly,  $\phi\in W^{1,2}_{0}(\Omega)$ is a weak solution of problem $(\ref{dirichlet})$ if and only if $\phi\in W^{1,2}_{0}(\Omega)$ is a critical point of $\mathcal{I}$. Moreover, it is easy to prove that any weak solution of \eqref{dirichlet} is also point-wise solution of \eqref{dirichlet}. We state the existence and uniqueness result as  follows.

\begin{theorem}\label{linear}
Let $\Omega$ be a non-empty finite subset of $V$ such that $\Omega^{c}$ is non-empty. Then for any function $f:\Omega\rightarrow \mathbb{R}$, the Dirichlet problem \eqref{dirichlet} has a unique solution.
\end{theorem}

In order to prove Theorem ~\ref{linear}, we present the Sobolev embedding in the following lemma.

\begin{lemma}\label{Sobolevembedding}
Let $\Omega$ be a non-empty finite subset of $V$ such that $\Omega^{c}$ is non-empty. Then $W_{0}^{1,2}(\Omega)$ is compactly embedded into $L^{s}(\Omega)$ for any $s\in [1,+\infty]$. In particular, there exists a constant $C$ depending only on $\Omega$ and $s$ such that for any $\phi\in W_{0}^{1,2}(\Omega)$,
\begin{equation*}
\|\phi\|_{s,\Omega}\leq C\|\phi\|_{W_{0}^{1,2}(\Omega)},
\end{equation*}
where $L^{s}(\Omega)$ is the linear space of functions $\phi:\Omega\rightarrow \mathbb{R} $ with the usual norm $\|\cdot\|_{s,\Omega}$.
Moreover, $W_{0}^{1,2}(\Omega)$ is  pre-compact.
\end{lemma}

\begin{proof}
 Since proof is similar to Theorem 7 in \cite{grigor2017existence}, we omit it here.
\end{proof}

\emph{The proof of Theorem \ref{linear}.}
\vskip4pt

We first prove the existence of solutions to \eqref{dirichlet} by using the direct variational method.
\vskip4pt
{\bf (i) $\mathcal{I}$ is weakly lower semi-continuous.}
\vskip4pt
For any $\phi\in W^{1,2}_{0}(\Omega)$, let
$$Q(\phi)=\int_{\Omega}f\phi d\mu.$$
Then
$$\mathcal{I}(\phi)=\frac{1}{2}\int_{\Omega\cup\partial\Omega}|\nabla\phi|^2d\mu-Q(\phi).$$
We claim that $Q$ is weakly continuous in $W^{1,2}_{0}(\Omega)$. In fact, let $\phi_{n}\rightharpoonup\phi$ in $W^{1,2}_{0}(\Omega)$ as $n\rightarrow\infty$. Then $\phi_{n}\rightharpoonup\phi$ in $L^2(\Omega)$.
Thus for any $\psi\in L^2(\Omega)$, we have
\begin{equation}\label{v0}
  \underset{n\rightarrow\infty}\lim\int_{\Omega}(\phi_{n}-\phi)\psi
  d\mu= \underset{n\rightarrow\infty}\lim\underset{v\in\Omega}\sum~\mu(v)(\phi_{n}(v)-\phi(v))\psi(v)=0.
\end{equation}
Take $v_0\in \Omega$ and
let
\begin{equation*}
	\psi_0(v)=\left\{
	\begin{aligned}
		&1, &&\text{if } v=v_0,\\
		&0, &&\text{if } v\neq v_{0}.
	\end{aligned}
    \right.
\end{equation*}
Obviously, $\psi_{0}$ belongs to $L^2(\Omega)$.
 By substituting $\psi_{0}$ into \eqref{v0}, we obtain
 \begin{equation*}
   \underset{n\rightarrow\infty}\lim\mu(v_0)(\phi_{n}(v_0)-\phi(v_0))=0,
 \end{equation*}
 which implies that $\underset{n\rightarrow\infty}\lim \phi_{n}(v)=\phi(v)$ for any $v\in \Omega.$ Thus
 \begin{equation*}
   \underset{n\rightarrow\infty}\lim [Q(\phi_{n})-Q(\phi)]= \underset{n\rightarrow\infty}\lim\int_{\Omega}f(v)(\phi_{n}(v)-\phi(v))d\mu=0.
 \end{equation*}
 Therefore, $Q$ is weakly continuous in $W^{1,2}_{0}(\Omega)$. Combining this with the weakly lower semi-continuity of the norm $\|\cdot\|_{W^{1,2}_0(\Omega)}$, we know that $\mathcal{I}$ is weakly lower semi-continuous.

\vskip4pt
{\bf (ii) $\mathcal{I}$ is coercive, i.e. $\mathcal{I}(\phi)\rightarrow\infty$ as $\|\phi\|_{W^{1,2}_0(\Omega)}\rightarrow\infty$.}
\vskip4pt

Note that $\Omega$ is finite. Then by Lemma~\ref{Sobolevembedding}, we have
	\begin{align}\label{coercive}
\mathcal{I}(\phi)&=\frac{1}{2}\|\phi\|^2_{W^{1,2}_0(\Omega)}-\int_{\Omega}f\phi d\mu\nonumber\\
&\geq \frac{1}{2}\|\phi\|^2_{W^{1,2}_0(\Omega)}-\int_{\Omega}|f\phi|d\mu\nonumber\\
&\geq \frac{1}{2}\|\phi\|^2_{W^{1,2}_0(\Omega)}-\underset{\Omega}\max |f|\int_{\Omega}|\phi|d\mu\nonumber\\
&\geq \frac{1}{2}\|\phi\|^2_{W^{1,2}_0(\Omega)}-\underset{\Omega}\max |f|C\|\phi\|_{W^{1,2}_0(\Omega)}\\
&\rightarrow +\infty\nonumber
\end{align}
as $\|\phi\|_{W_{0}^{1,2}(\Omega)}\rightarrow\infty.$
\vskip4pt
{\bf (iii) $\mathcal{I}$ is bounded from below.}
\vskip4pt
It follows from \eqref{coercive} that
\begin{align*}
\mathcal{I}(\phi)&\geq \frac{1}{2}\|\phi\|^2_{W^{1,2}_0(\Omega)}-\underset{\Omega}\max |f|C\|\phi\|_{W_0^{1,2}(\Omega)}\\
&=\frac{1}{2}\left(\|\phi\|_{W^{1,2}_0(\Omega)}-\underset{\Omega}\max |f|C\right)^2-\frac{1}{2}\left(\underset{\Omega}\max |f|C\right)^2\\
&\geq-\frac{1}{2}\left(\underset{\Omega}\max |f|C\right)^2,
\end{align*}
which implies that $\mathcal{I}$ is bounded from below.

Combining (i), (ii) and (iii), we obtain that there exists $\phi^*\in W^{1,2}_{0}(\Omega)$ such that 
$$\mathcal{I}(\phi^*)=\underset{\phi\in W^{1,2}_{0}(\Omega)}\inf \mathcal{I}(\phi)$$
and thus $\phi^*$ is a solution of \eqref{dirichlet}.

Now, we prove the uniqueness of the solution to \eqref{dirichlet} by contradiction. Here we stipulate that if $\phi=0$ on $\partial\Omega$, then $\phi=0$ on $\Omega^c$. Without loss of generality, we assume that there are two different solutions $\phi_1$ and $\phi_2$ of \eqref{dirichlet}. Then $\phi=\phi_1-\phi_2$ satisfies
\begin{align*}
\begin{cases}
-\Delta \phi=0, \  &\text{in}\ \ \Omega,\\
\phi=0,\  &\text{on} \ \  \partial\Omega.
\end{cases}
\end{align*}
It follows from Lemma~\ref{maxmin} that $\phi\equiv 0$, and thus $\phi_1=\phi_2$. We complete the proof of Theorem~\ref{linear}.

\subsection{Yamabe type equation}

In this subsection, we consider the existence of solutions to the following Yamabe equation
\begin{align}\label{Yamabe}
\begin{cases}
-\Delta \phi-\alpha \phi=|\phi|^{p-2}\phi, \  &\text{in}\ \ \Omega,\\
\phi=0,\  &\text{on} \ \  \partial\Omega,
\end{cases}
\end{align}
where $H=(V,E)$ is a connected finite hypergraph, $\Omega$ is a non-empty finite subset of $V$ such that $\Omega^{c}$ is non-empty, $p>2$ and 
$$\alpha<\lambda_{1}(\Omega)=\underset{\phi\not\equiv 0,\phi|_{\partial\Omega}=0}\inf\frac{\int_{\Omega\cup\partial\Omega}|\nabla \phi|^2d\mu}{\int_{\Omega}\phi^2d\mu}.$$ 

We state the existence result as follows.
\begin{theorem}\label{existence}
If $\alpha<\lambda_{1}(\Omega)$, then the equation \eqref{Yamabe} admits a nontrivial solution for any $p>2$.
\end{theorem}

It is suitable to study the equation \eqref{Yamabe} on the space $W^{1,2}_{0}(\Omega)$. The functional related to \eqref{Yamabe} is defined by
\begin{equation*}
	\mathcal{E}(\phi)=\frac{1}{2}\int_{\Omega\cup\partial\Omega}|\nabla \phi|^2d\mu-\frac{1}{2}\int_{\Omega}\alpha \phi^2d\mu-\frac{1}{p}\int_\Omega |\phi|^pd\mu,\ \ \phi\in W^{1,2}_0(\Omega).
\end{equation*}
For any $\psi\in W^{1,2}_0(\Omega)$,
\begin{equation*}
\langle \mathcal{E}^\prime(\phi),\psi\rangle=\int_{\Omega\cup\partial\Omega}\nabla \phi\nabla\psi d\mu-\alpha\int_\Omega \phi\psi d\mu-\int_\Omega|\phi|^{p-2}\phi\psi d\mu,\ \ \phi\in W^{1,2}_0(\Omega).
\end{equation*}
Clearly, $\phi\in W^{1,2}_0(\Omega)$ is a weak solution of the equation $\eqref{Yamabe}$ if and only if $\phi$ is a critical point of $\mathcal{E}$.

\vskip8pt
\emph{The proof of Theorem \ref{existence}.}~~Our proof is divided into three steps.
\vskip4pt

\textbf{Step 1. There exist two constants $\xi>0$ and $\rho>0$ such that $\mathcal{E}(\phi)\geq \xi$ for all $\phi$ with $\|\phi\|_{W^{1,2}_0(\Omega)}=\rho$.}

Since $\alpha<\lambda_{1}(\Omega)$, there exists $\tau>0$ such that $\alpha\leq \lambda_1(\Omega)-\tau$. By simple calculations, we have
\begin{align*}
  \mathcal{E}(\phi) &= \frac{1}{2}\|\phi\|^2_{W^{1,2}_0(\Omega)}-\frac{1}{2}\alpha\|\phi\|^2_{2,\Omega}-\frac{1}{p}\|\phi\|^{p}_{p,\Omega} \\
   &\geq\frac{1}{2}\|\phi\|^2_{W^{1,2}_0(\Omega)}-\frac{\lambda_{1}(\Omega)-\tau}{2}\|\phi\|^2_{2,\Omega}-\frac{1}{p}\|\phi\|^{p}_{p,\Omega}\\
   &\geq\frac{1}{2}\|\phi\|^2_{W^{1,2}_0(\Omega)}-\frac{\lambda_{1}(\Omega)-\tau}{2\lambda_{1}(\Omega)}\|\phi\|^2_{W^{1,2}_0(\Omega)}
   -\frac{1}{p}C\|\phi\|^{p}_{W^{1,2}_0(\Omega)}\\
   &=\frac{\tau}{2\lambda_{1}(\Omega)}\|\phi\|^2_{W^{1,2}_0(\Omega)}-\frac{1}{p}C\|\phi\|^{p}_{W^{1,2}_0(\Omega)}\\
   &=\|\phi\|^2_{W^{1,2}_0(\Omega)}\left(\frac{\tau}{2\lambda_{1}(\Omega)}-\frac{1}{p}C\|\phi\|^{p-2}_{W^{1,2}_0(\Omega)}\right).
\end{align*}
Taking $\rho=\left(\frac{p\tau}{4\lambda_{1}(\Omega)C}\right)^{\frac{1}{p-2}}$, then we have
\begin{equation*}
  \mathcal{E}(\phi)\geq \frac{\tau}{4\lambda_{1}(\Omega)}\rho^2:=\xi
\end{equation*}
for all $\phi\in W^{1,2}(\Omega)$ satisfying $\|\phi\|_{W^{1,2}(\Omega)}=\rho$.

\vskip4pt
\textbf{Step 2. There exists some non-negative function $\phi\in W^{1,2}_0(\Omega)$ such that $\mathcal{E}(t\phi)\ra-\infty$ as $t\ra+\infty$.}
\vskip4pt
Let $v_0\in \Omega$ be a fixed point and
\begin{equation*}
	\phi(v)=\left\{
	\begin{aligned}
		&1, &&\text{if } v=v_0,\\
		&0, &&\text{if } v\neq v_{0}.
	\end{aligned}
    \right.
\end{equation*}
Then we have
\begin{align*}
\mathcal{E}(t\phi)&=\f{t^2}{2}\|\phi\|^2_{W^{1,2}_0(\Omega)}-\f{t^2}{2}\alpha\|\phi\|^2_{2,\Omega}-\f{t^p}{p}\|\phi\|^p_{p,\Omega}\\
&=\f{t^2}{2}\sum_{v\in \Omega\cup\partial\Omega}\sum_{e\in E:v\in e}\frac{\omega_e}{\delta_e(\delta_e-1)}\left[\sum_{u\in e}
(\phi(u)-\phi(v))\right]^2-\f{t^2}{2}\alpha\mu(v_{0})
-\f{t^p}{p}\mu(v_{0})\\
&\ra-\infty
\end{align*}
 as $t\ra+\infty$, since $\Omega$ is finite, $p>2$ and $\mu(v_0)>0$.

\vskip4pt
\textbf{Step 3. $\mathcal{E}$ satisfies the $(PS)_c$ condition.}
\vskip4pt

Let $\{\phi_k\}\subset W^{1,2}_0(\Omega)$ be a sequence satisfying $\mathcal{E}(\phi_k)\ra c$ and $\mathcal{E}^\prime(\phi_k)\ra 0$ as $k\rightarrow \infty$. Note that $\alpha<\lambda_{1}(\Omega)$, then there exists $\tau>0$ such that $\alpha\leq \lambda_{1}(\Omega)-\tau$.
Thus we have
\begin{align*}
c+1+\|\phi_k\|_{W^{1,2}_0(\Omega)} &\geq \mathcal{E}(\phi_k)-\frac{1}{p}\langle \mathcal{E}^\prime(\phi_k),\phi_k \rangle \\
 &= \left(\frac{1}{2}-\frac{1}{p}\right)\|\phi_k\|^2_{W^{1,2}_0(\Omega)}
 -\left(\frac{1}{2}-\frac{1}{p}\right)\alpha\|\phi_k\|^2_{2,\Omega}\\
 &\geq\left(\frac{1}{2}-\frac{1}{p}\right)\|\phi_k\|^2_{W^{1,2}_0(\Omega)}
 -\left(\frac{1}{2}-\frac{1}{p}\right)(\lambda_{1}(\Omega)-\tau)\|\phi_k\|^2_{2,\Omega}\\
 &\geq\left(\frac{1}{2}-\frac{1}{p}\right)\|\phi_k\|^2_{W^{1,2}_0(\Omega)}
 -\left(\frac{1}{2}-\frac{1}{p}\right)\frac{\lambda_{1}(\Omega)-\tau}{\lambda_{1}(\Omega)}\|\phi_k\|^2_{W^{1,2}_0(\Omega)}\\
 &=\left(\frac{1}{2}-\frac{1}{p}\right)\frac{\tau}{\lambda_{1}(\Omega)}\|\phi_k\|^2_{W^{1,2}_0(\Omega)},
 \end{align*}
which implies that $\{\phi_k\}$ is bounded in $W^{1,2}_0(\Omega)$. By Lemma \ref{Sobolevembedding}, up to a subsequence, there exists some $\phi\in W^{1,2}_0(\Omega)$ such that $\phi_k\rightarrow \phi$ in $W^{1,2}_0(\Omega)$ as $k\rightarrow \infty$.

By Steps 1, 2 and 3, $\mathcal{E}$ satisfies all the hypotheses of
the mountain pass theorem in \cite{willem1997minimax}. Then we conclude that
  $$\hat{c}=\min_{\gamma\in\Gamma}\max_{t\in[0,1]}\mathcal{E}(\gamma(t))$$
  is the critical level of $\mathcal{E}$, where
  $$\Gamma=\le\{\g\in C([0,1],W^{1,2}_0(\Omega)): \g(0)=0, \mathcal{E}(\g(1))<0\ri\}.$$
Thus, there exists $\phi^{*}\in W^{1,2}_0(\Omega)$ satisfying $\mathcal{E}(\phi^{*})=\hat{c}\geq \xi>0$ and $\mathcal{E}'(\phi^{*})=0$. Therefore, $\phi^{*}$ is a  nontrivial solution of equation \eqref{Yamabe}.

\subsection{Nonlinear Schr\"{o}dinger equation}

In this subsection, we consider the existence of positive solutions to the following nonlinear Schr\"{o}dinger equation
\begin{equation}\label{schrodinger}
  -\Delta \phi+h(v)\phi=f(v,\phi)
\end{equation}
on a connected finite hypergraph $H$, where $h$ is a positive function defined on $V$. Since $V$ is finite, there exists some constant $h_0>0$ such that $h\geq h_0$ for all $v\in V$.  The nonlinear term $f:V\times \mathbb{R}\rightarrow\mathbb{R}$ satisfies the following hypotheses:
\begin{itemize}
  \item [$(f_1)$] \emph{For any $v\in V$, $f(v,s)$ is continuous in $s\in \mathbb{R}$.}
  \item [$(f_2)$] \emph{For all $(v,s)\in V\times [0,+\infty)$, $f(v,s)\geq0$ and $f(v,0)=0$ for all $v\in V$.}
  \item [$(f_3)$] \emph{There exist $\theta>2$ and $s_0>0$ such that if $s\geq s_0$, then there holds}
\begin{equation*}
  F(v,s)=\int^{s}_{0}f(v,t)dt\leq \frac{1}{\theta}sf(v,s), \ \ \forall v\in V.
\end{equation*}
\item  [$(f_4)$] \emph{For any $v\in V$, there holds}
\begin{equation*}
  \underset{s\rightarrow0^{+}}\limsup\frac{f(v,s)}{s}<\lambda_{1}(V)=\underset{\phi\not\equiv0}\inf\frac{\int_{V}
  (|\nabla \phi|^2+h(v)\phi^2)d\mu}{\int_{V}\phi^2d\mu}.
\end{equation*}
\end{itemize}

Our main result is as follows.

\begin{theorem}\label{pos-existence}
Let $H$ be a connected finite hypergraph. Assume that $(f_1)-(f_4)$ hold. Then the equation \eqref{schrodinger} admits a positive solution.
\end{theorem}

In order to obtain the positive solution of \eqref{schrodinger}, we consider the following equation
\begin{equation}\label{pos}
  -\Delta \phi+h(v)\phi=f(v,\phi^+)\ \ \hbox{in} \ \ V,
\end{equation}
where $\phi^{+}=\max\{\phi,0\}$. Then by Lemma~\ref{Wmp} (the weak maximum principle), we can obtain the following conclusion.

\begin{lemma}\label{pos1} If $f$ satisfies $(f_{2})$ and $\phi$ is a nontrivial solution of the equation \eqref{pos}, then $\phi$ is a strictly positive solution of the equation \eqref{schrodinger}.
\end{lemma}

\begin{proof}
Suppose that $\phi$ is a nontrivial solution of \eqref{pos}. By $(f_2)$, we know that $f(v,\phi^{+})\geq 0$. Then applying Lemma~\ref{Wmp} to the equation \eqref{pos}, we have $\phi(v)\geq 0$ for all $v\in V$. We now prove $\phi(v)>0$ for all $v\in V$. Suppose not, there exists a point $v^*\in V$ such that $\phi(v^*)=0=\underset{v\in V}\min\phi$ and $\Delta\phi(v^*)>0$. It follows that
\begin{equation*}
  0>-\Delta\phi(v^*)=f(v^*,\phi(v^*))=0,
\end{equation*}
which is a contradiction, and implies $\phi(v)>0$ for all $v\in V$. Thus $\phi$ is a positive solution of \eqref{schrodinger}.
\end{proof}

According to Lemma \ref{pos1}, to prove Theorem~\ref{pos-existence}, we only need to prove that the equation \eqref{pos} has a nontrivial solution. Now, we will prove the existence of nontrivial solutions to \eqref{pos} by using the mountain pass theorem.

Let $W^{1,2}(V)$ be defined as a set of all functions $\phi:V\rightarrow \mathbb{R}$ under the norm
\begin{equation}\label{norm}
  \|\phi\|_{W^{1,2}(V)}=\left(\int_V(|\nabla \phi|^2+h(v)\phi^2)d\mu\right)^{1/2}.
\end{equation}
Since $h$ is a positive function and $V$ is a finite set, \eqref{norm} is a norm equivalent to the standard norm $\|\phi\|=\left(\int_V(|\nabla \phi|^2+\phi^2)d\mu\right)^{1/2}$ on $W^{1,2}(V)$.
Clearly, the space $W^{1,2}(V)$ is a Hilbert space with its inner product
$$\langle\phi,\psi\rangle_{W^{1,2}(V)}=\int_{V}(\nabla\phi\nabla\psi+h(v)\phi\psi)d\mu,\ \ \forall \phi, \psi\in W^{1,2}(V).$$
The corresponding energy functional of the equation \eqref{pos} is defined by 
\begin{equation*}
  \mathcal{J}(\phi)=\frac{1}{2}\int_{V}(|\nabla\phi|^2+h(v)\phi^2) d\mu-\int_{V}F(v,\phi^{+})d\mu, \ \ \phi\in W^{1,2}(V).
\end{equation*}
It is easy to verify that $\mathcal{J}\in C^{1}(W^{1,2}(V),\mathbb{R})$ and for any $\psi\in W^{1,2}(V)$, 
\begin{equation*}
\langle \mathcal{J}^\prime(\phi),\psi\rangle=\int_{V}(\nabla\phi\nabla\psi +\phi\psi) d\mu-\int_V f(v,\phi^{+})\psi d\mu.
\end{equation*}
Obviously, $\phi\in W^{1,2}(V)$ is a weak solution of the equation $\eqref{pos}$ if and only if $\phi\in W^{1,2}(V)$ is a critical point of $\mathcal{J}$.  Here, we give the following Sobolev embedding theorem, which will be used later.

\begin{lemma}\label{sobolevembedding}
Let $H$ be a connected finite hypergraph. Then $W^{1,2}(V)$ is compactly embedded into $L^{q}(V)$ for any $q\in [1,+\infty]$. In particular, there exists a constant $C$ depending only on $V$, $\mu_{\min}$ and $q$ such that for any $\phi\in W^{1,2}(V)$,
\begin{equation}\label{Sin}
\|\phi\|_{q}\leq C\|\phi\|_{W^{1,2}(V)},
\end{equation}
where $\mu_{\min}=\underset{v\in V}\min\mu(v)$. Moreover, $W^{1,2}(V)$ is  pre-compact.
\end{lemma}

\begin{proof} Since $V$ is a finite set, $W^{1,2}(V)$ is a finite dimensional space. Hence the conclusions of the lemma obviously hold.
\end{proof}

Next, we prove that $\mathcal{J}$ satisfies the mountain pass geometry.

\begin{lemma}\label{mp1} Assume that $(f_1)-(f_4)$ hold. Then $\mathcal{J}$ satisfies the mountain pass geometry. Namely,
\begin{enumerate}[\upshape (i)]
  \item  there exist positive constants $\delta$,$r$ such that $\mathcal{J}\geq\delta$ for all functions $\phi$ with $||\phi||_{W^{1,2}(V)}=r$;
  \item  there exists some non-negative function $\phi\in W^{1,2}(V)$ such that $\mathcal{J}(t\phi)\rightarrow -\infty$ as $t\rightarrow+\infty.$
\end{enumerate}
\end{lemma}
\begin{proof}
(i) By $(f_{4})$, there exist two positive constants $\tau$ and $\sigma$ such that  
\begin{equation*}
  F(v,\phi^{+})\leq\frac{\lambda_{1}(V)-\tau}{2}(\phi^+)^2+\frac{(\phi^+)^3}{\sigma^3}F(v,\phi^+).
\end{equation*}
For any $\phi\in W^{1,2}(V)$ with $\|\phi\|_{W^{1,2}(V)}\leq 1$, by Lemma~\ref{sobolevembedding}, we have
$\|\phi\|_{L^{\infty}(V)}\leq C$ for some constant $C$ depending only on $V$ and $\mu_{\min}$. Then
\begin{equation*}
  F(v,\phi^{+})\leq\frac{\lambda_{1}(V)-\tau}{2}(\phi^+)^2+C_1(\phi^+)^3.
\end{equation*}
Thus we obtain
\begin{align*}
  \mathcal{J}(\phi) &= \frac{1}{2}\|\phi\|^{2}_{W^{1,2}(V)}-\int_{V}F(v,\phi^{+}) \\
   &\geq  \frac{1}{2}\|\phi\|^{2}_{W^{1,2}(V)}-\frac{\lambda_{1}(V)-\tau}{2}\int_{V}(\phi^+)^2d\mu-C_1\int_{V}(\phi^+)^3d\mu\\
   &\geq  \frac{1}{2}\|\phi\|^{2}_{W^{1,2}(V)}-\frac{\lambda_{1}(V)-\tau}{2}\int_{V}\phi^2d\mu-C_1\int_{V}|\phi|^3d\mu\\
   &\geq  \frac{1}{2}\|\phi\|^{2}_{W^{1,2}(V)}-\frac{\lambda_{1}(V)-\tau}{2\lambda_{1}(V)}\|\phi\|^{2}_{W^{1,2}(V)}
   -C_{1}C_{2}\|\phi\|^{3}_{W^{1,2}(V)}\\
   &=\|\phi\|^{2}_{W^{1,2}(V)}\left(\frac{\tau}{2\lambda_{1}(V)}-C_1C_2\|\phi\|_{W^{1,2}(V)}\right).
\end{align*}
Taking $r:=\min\{1,\frac{\tau}{4\lambda_{1}(V)C_{1}C_{2}}\}$, then we get
\begin{equation*}
  \mathcal{J}(\phi)\geq\frac{\tau}{4\lambda_{1}(V)}r^2:=\delta>0
\end{equation*}
for all $\phi\in W^{1,2}(V)$ satisfying $||\phi||_{W^{1,2}(V)}=r$.
\vskip4pt
(ii) By $(f_{3})$, there exist two positive constants $c_1$ and $c_2$ such that
\begin{equation*}
  F(v,\phi^{+})\geq c_{1}(\phi^{+})^{\theta}-c_2.
\end{equation*}
Fixed $v_0\in V$, let
\begin{equation*}
	\phi(v)=\left\{
	\begin{aligned}
		&1, &&\text{if } v=v_0,\\
		&0, &&\text{if } v\neq v_{0}.
	\end{aligned}
    \right.
\end{equation*}
Then we have
\begin{align*}
  \mathcal{J}(t\phi) &= \frac{t^2}{2}\int_{V}(|\nabla\phi|^2+h(v)\phi^2)d\mu-\int_{V}F(v,t\phi^{+})d\mu \\
  &\leq\f{t^2}{2}\sum_{v\in V}\sum_{e\in E:v\in e}\frac{\omega_e}{\delta_e(\delta_e-1)}\left[\sum_{u\in e}
(\phi(u)-\phi(v))\right]^2+\frac{t^2}{2}h(v_0)\mu(v_0)-c_1\mu(v_0)t^{\theta}+c_2\mu(v_{0})\\
   &\rightarrow -\infty
   \end{align*}
as $t\ra+\infty$, since $V$ is finite, $\theta>2$ and $\mu(v_0)>0$.
\end{proof}

Finally, we verify that $\mathcal{J}$ satisfies the $(PS)_c$ condition.

\begin{lemma}\label{ps}
$\mathcal{J}$ satisfies the $(PS)_c$ condition for any $c\in\mathbb{R}$.
\end{lemma}
\begin{proof}
Let $\{\phi_k\}\subset W^{1,2}(V)$ be a sequence satisfying $\mathcal{J}(\phi_k)\ra c$ and $\mathcal{J}^\prime(\phi_k)\ra 0$ as $k\rightarrow \infty$. Then we obtain
\begin{equation*}
   \frac{1}{2}\int_{V}(|\nabla\phi_k|^2+h(v)\phi_k^2)d\mu-\int_{V}F(v,\phi_{k}^{+})d\mu=c+o_{k}(1)
\end{equation*}
and
\begin{equation*}
  \int_{V}(|\nabla\phi_k|^2+h(v)\phi_k^2)d\mu-\int_{V}f(v,\phi^{+}_{k})\phi_kd\mu=o_{k}(1)\|\phi_{k}\|_{W^{1,2}(V)}.
\end{equation*}
It follows from $(f_3)$ and the above two equations that $\{\phi_k\}$ is bounded in $W^{1,2}(V)$. By Lemma \ref{sobolevembedding}, up to a subsequence, there exists some $\phi\in W^{1,2}(V)$ such that $\phi_k\rightarrow \phi$ in $W^{1,2}(V)$ as $k\rightarrow \infty$.
\end{proof}

 \emph{Proof of Theorem~\ref{pos-existence}.} ~~By Lemma~\ref{mp1} and Lemma~\ref{ps}, we know that $\mathcal{J}$ satisfies all the hypotheses of the mountain pass theorem in \cite{willem1997minimax}. Then we conclude that
$$c=\min_{\gamma\in\Gamma}\max_{t\in[0,1]}\mathcal{J}(\gamma(t))$$
is the critical level of $J$, where
$$\Gamma=\le\{\g\in C([0,1],W^{1,2}(V)): \g(0)=0, \mathcal{J}(\g(1))<0\ri\}.$$
Thus, there exists $\phi\in W^{1,2}(V)$ satisfying $\mathcal{J}(\phi)=c\geq \delta>0$ and $\mathcal{J}'(\phi)=0$. Thus, $\phi$ is a nontrivial solution of \eqref{pos}. It follows form Lemma~\ref{pos1} that $\phi$ is a positive solution of \eqref{schrodinger}. We complete the proof of Theorem~\ref{pos-existence}.

\section*{Acknowledgements}
This research is supported by National Natural Science Foundation of China (No. 12271039 and No. 12101355), the National Key R and D Program of China (No. 2020YFA0713100) and the Open Project Program (No. K202303) of Key Laboratory of Mathematics and Complex Systems, Beijing Normal University.
%\section*{References}

%\bibliographystyle{elsarticle-num}
%\bibliographystyle{aipauth4-1}

\bibliographystyle{elsarticle-num-names-alpha}

\bibliography{mybib-hypergraph}

\end{document}